\documentclass[reqno]{amsart}
\usepackage{amssymb,latexsym,amsmath,amsthm,enumerate,amsbsy}
\usepackage[mathscr]{eucal}
\usepackage{framed,color,graphicx}
\usepackage{mathrsfs}
\usepackage[all]{xy}
\usepackage{tikz}
\usetikzlibrary{positioning,decorations.pathreplacing,patterns,decorations.pathmorphing}
\tikzset{%
element/.style={draw, shape=circle, fill=white, inner sep=1.4pt}
}

\DeclareSymbolFont{bbold}{U}{bbold}{m}{n}
\DeclareSymbolFontAlphabet{\mathbbold}{bbold}

\theoremstyle{plain}
\newtheorem{thm}{Theorem}[section]
\newtheorem{lem}[thm]{Lemma}
\newtheorem{cor}[thm]{Corollary}
\newtheorem{pro}[thm]{Proposition}

\theoremstyle{definition}

\newtheorem{remark}[thm]{Remark}

\newcommand{\up}[1]{\textup{#1}}

\newcommand{\bp}{\mathbf{p}}
\newcommand{\bq}{\mathbf{q}}
\newcommand{\br}{\mathbf{r}}

\newcommand{\bt}{\mathbf{t}}
\newcommand{\bu}{\mathbf{u}}
\newcommand{\bv}{\mathbf{v}}
\newcommand{\bw}{\mathbf{w}}

\begin{document}

\title[The finite basis problem for subvarieties of $\mathsf{V}(S_7^0)$]
{The finite basis problem for all subvarieties of the variety $\mathsf{V}(S^0_7)$}

\author{Mengya Yue}
\address{School of Mathematics, Northwest University, Xi'an, 710127, Shaanxi, P.R. China}
\email{myayue@yeah.net}
\date{}
\subjclass[2010]{16Y60, 03C05, 08B15}
\keywords{Additively idempotent semiring, finite basis problem, subvariety lattice}

\begin{abstract}
We solve the finite basis problem for all subvarieties of
$\mathsf{V}(S^0_7)$. By considering their intersections with
$\mathsf{V}(S_7)$, we reduce the problem to six intervals. We show that
$\mathsf{V}(S_7^0)$ has exactly fifteen finitely based subvarieties and
that every other subvariety is nonfinitely based. Moreover,
$\mathsf{V}(S_c(abc))$ and $\mathsf{V}(SR_6)$ are the only limit
subvarieties of $\mathsf{V}(S_7^0)$.
\end{abstract}

\maketitle

\section{Introduction}
A variety is called \emph{finitely based} if it can be defined by a
finite set of identities; otherwise, it is called \emph{nonfinitely
based}. An algebra is said to be finitely based or nonfinitely based
according as the variety it generates has the corresponding property.
The finite basis problem, which asks whether a given algebra or variety
admits a finite equational basis, is one of the classical problems of
universal algebra. It has been studied extensively for groups,
semigroups, rings and semirings; see, for example,
\cite{gk,hn,np,vol01,jrz,rjzl}.

The finite basis problem becomes more structural when it is considered
simultaneously for all members of a subvariety lattice. A variety is
\emph{hereditarily finitely based} if all its subvarieties are finitely
based. At the opposite boundary are the \emph{limit varieties}: these
are nonfinitely based varieties all of whose proper subvarieties are
finitely based, or equivalently, the minimal nonfinitely based
varieties. Thus the determination of the finitely based members of a
subvariety lattice also identifies its limit subvarieties. This point of
view has proved useful in the study of semigroup and semiring varieties;
see \cite{ks,lv,gu,gulz,gus,j,sa,zl}.

An \emph{additively idempotent semiring}, or an \emph{ai-semiring}, is
an algebra $(S,+,\cdot)$ in which $(S,+)$ is a commutative idempotent
semigroup, $(S,\cdot)$ is a semigroup, and multiplication distributes
over addition. Thus
\[
x(y+z)\approx xy+xz,\quad (x+y)z\approx xz+yz
\]
hold. It is \emph{commutative} if its multiplicative reduct is
commutative. Ai-semirings occur naturally in algebraic geometry,
tropical geometry, information science and theoretical computer
science; see \cite{cc,ms,gl,go}. Their addition induces the natural
order
\[
a\leq b\quad\Longleftrightarrow\quad a+b=b,
\]
which is compatible with both operations. The finite basis problem for
ai-semirings has developed rapidly in recent years, particularly for
small finite ai-semirings and for flat semirings; see
\cite{d,gmrz,gpz05,jrz,mr,pas05,rlzc,rlyc,rjzl,rzw,rzs20,
shap23,sr,wrz,yrzs,zrc}. Among the relevant results, Wu, Ren and
Zhao~\cite{wrz} proved that $S_7^0$ is nonfinitely based, while
Gao, Jackson, Ren and Zhao~\cite{gmrz} investigated the finite basis
problem for ai-semirings related to $S_7$ and determined the
subvariety lattice of $\mathsf{V}(S_7)$. Ren, Jackson, Zhao and
Lei~\cite[Theorem~3.6]{rjzl} proved that
$\mathsf{V}(S_c(abc))$ is a limit variety. These results naturally
lead to the problem considered here: determine the finite basis status
of every subvariety of $\mathsf{V}(S_7^0)$.

A second ingredient is the six-element commutative ai-semiring $SR_6$
whose operations are displayed in Table~\ref{6-element ai-semirings}.
Lyu, Ren and Yue~\cite{lry} proved that $\mathsf{V}(SR_6)$ is a
limit variety. In the present paper, $SR_6$ also provides the lower
boundary of the two nonfinitely based intervals that arise in our
classification.

\begin{table}[htbp]
\caption{The Cayley tables of $SR_6$}\label{6-element ai-semirings}
\begin{tabular}{c|cccccc}
$+$               &1 & 2 & 3 & 4 & 5 & 6\\
\hline
                 1& 1 & 1 & 1 & 1 & 1 & 1 \\
                 2&1 & 2 & 1 & 1 & 2 & 1\\
                 3&1 & 1 & 3 & 1 & 1 & 1\\
                 4&1 & 1 & 1 & 4 & 1 & 4\\
                 5&1 & 2 & 1 & 1 & 5 & 1\\
                 6&1 & 1 & 1 & 4 & 1 & 6
\end{tabular}
\qquad\qquad
\begin{tabular}{c|cccccc}
$\cdot$               &1 & 2 & 3 & 4 & 5 & 6\\
\hline
                   1&1 & 1 & 1 & 1 & 1 & 1 \\
                   2&1 & 1 & 1 & 1 & 1 & 3\\
                   3&1 & 1 & 1 & 1 & 1 & 1\\
                   4&1 & 1 & 1 & 1 & 3 & 1\\
                   5&1 & 1 & 1 & 3 & 1 & 3\\
                   6&1 & 3 & 1 & 1 & 3 & 1
\end{tabular}
\end{table}

Our approach combines equational bases, omission identities and the
known lattice $\mathcal L(\mathsf{V}(S_7))$. For
$\mathcal V\leq\mathsf{V}(S_7^0)$, consider
\[
\varphi(\mathcal V)=\mathcal V\cap\mathsf{V}(S_7).
\]
The subvarieties of $\mathsf{V}(S_7)$ that do not contain
$S_c(abc)$ give rise to six fibres of $\varphi$. We determine these
fibres using finite equational bases and identities characterizing the
omission of the relevant small semirings. This reduces the finite basis
problem for $\mathsf{V}(S_7^0)$ to four finite fibres and two
intervals.

The main result states that $\mathsf{V}(S_7^0)$ has exactly fifteen
finitely based subvarieties. Every other subvariety is nonfinitely
based. More precisely, apart from the subvarieties containing
$S_c(abc)$, the nonfinitely based varieties are exactly the nonleast
members of the intervals
\[
[\mathsf{V}(S_c(ab)),\mathsf{V}(S_c(ab)^0)]
\quad\text{and}\quad
[\mathsf{V}(M_2,S_c(ab)),
  \mathsf{V}(M_2^0,S_c(ab)^0)].
\]
The proof uses the fact that every nonleast member of either interval
contains $SR_6$, together with a finite relative basis of
$\mathsf{V}(SR_6)$ inside $\mathsf{V}(S_7^0)$. As a further
consequence, $\mathsf{V}(S_c(abc))$ and $\mathsf{V}(SR_6)$ are the
only limit subvarieties of $\mathsf{V}(S_7^0)$.

The paper is organized as follows. Section~2 recalls the required
terminology and equational tools. In Section~3 we establish the
equational bases and omission criteria, determine the six fibres, and
prove the finite basis classification. The final section records the
resulting description of the finitely based and limit subvarieties.

\section{Preliminaries}

We fix the notation used throughout the paper.  All semirings under
consideration are commutative ai-semirings.  For a class $\mathcal K$
of such semirings, $\mathsf{V}(\mathcal K)$ denotes the variety
generated by $\mathcal K$.  When $\mathcal K=\{S_1,\ldots,S_k\}$, we
write $\mathsf{V}(S_1,\ldots,S_k)$, and $\mathbf T$ denotes the trivial
variety.

Let $X$ be a countably infinite set of variables and let $X_c^+$ be
the free commutative semigroup on $X$.  Elements of $X_c^+$ will be
called \emph{words}.  A term in the language of commutative
ai-semirings may be identified with a finite nonempty subset of
$X_c^+$.  We write such a term as
\[
\bu=\bu_1+\cdots+\bu_n,
\qquad \bu_i\in X_c^+.
\]
Since addition is idempotent and commutative, neither the order nor
the repetition of the summands matters.  We use bold letters for
words and terms, and ordinary letters $x,y,z,\ldots$ for variables.

Let $P_f(X_c^+)$ be the set of all finite nonempty subsets of
$X_c^+$.  With union as addition and setwise multiplication as
multiplication, $P_f(X_c^+)$ is the free commutative ai-semiring on
$X$; see~\cite[Theorem~2.5]{ku}.  Thus a substitution is simply an
endomorphism of $P_f(X_c^+)$.  An identity $\bu\approx\bv$ holds in a
commutative ai-semiring $S$ if
$\varphi(\bu)=\varphi(\bv)$ for every homomorphism
$\varphi:P_f(X_c^+)\to S$.

The addition of an ai-semiring induces its natural order:
\[
a\leq b \quad\Longleftrightarrow\quad a+b=b.
\]
Accordingly, we use
\[
\bu\preceq\bv
\quad\text{and}\quad
\bv\succeq\bu
\]
as abbreviations for the identity $\bv\approx\bv+\bu$.  We call
either expression an \emph{inequality}.  An identity
$\bu\approx\bv$ is equivalent to the two inequalities
$\bu\preceq\bv$ and $\bv\preceq\bu$.  Moreover, if
\[
\bu=\bu_1+\cdots+\bu_k,
\qquad
\bv=\bv_1+\cdots+\bv_\ell,
\]
then $\bu\approx\bv$ is equivalent to
\[
\bu_i\preceq\bv\quad(1\leq i\leq k),
\qquad
\bv_j\preceq\bu\quad(1\leq j\leq\ell).
\]
It will therefore be enough to work with inequalities
$\bq\preceq\bu$ whose left-hand side is a single word.

Next, we introduce some notation.
Let $\bw$ be a nonempty word, and let $x$ be a variable. Then
\begin{itemize}
\item $c(\bw)$ denotes the \emph{content} of $\bw$, that is, the set of all variables that occur in $\bw$;

\item $\ell(\bw)$ denotes the \emph{length} of $\bw$, that is, the number of variables occurring in $\bw$ counting multiplicities;
\end{itemize}

Now let $\bu$ be a term such that $\bu=\bu_1+\bu_2+\cdots+\bu_n$,
where $\bu_i \in X^+$, $1 \leq i \leq n$.
Let $\bq$ be a nonempty word, and let $k$ be a positive integer. Then
\begin{itemize}
\item $c(\bu)$ denotes the \emph{content} of $\bu$, that is,
\[
c(\bu)=\bigcup_{i=1}^n c(\bu_i);
\]

\item $L_{\geq k}(\bu)$ denotes the set $\{\bu_i \in \bu \mid \ell(\bu_i)\geq k\}$;

\item $D_{\bq}(\bu)$ denotes the set $\{\bu_i \in \bu \mid c(\bu_i)\subseteq c(\bq)\}$;

\item $\delta(\bu)$ denotes the set of subsets $Z$ of $c(\bu)$ such that for every
 $\bu_i\in\bu$, $Z\cap c(\bu_i)$ is a singleton and $occ(x,\bu_i)=1$ if $\{x\}=Z\cap c(\bu_i)$.
\end{itemize}

For later use, we also recall the notion of freeness introduced in
\cite{yrg}.  A term $\bu$ is a \emph{subterm} of a term $\bv$ if
\[
\bv=\bp\bu+\br
\]
for suitable $\bp$ and $\br$, where $\bp$ may be the empty word and
$\br$ the empty sum.  A term $\bv$ is $\bu$-\emph{free} if no
substitution instance of $\bu$ is a subterm of $\bv$.  Notice that if
$\bu$ is a subterm of $\bw$ and $\bv$ is $\bu$-free, then $\bv$ is
also $\bw$-free.  Indeed, any substitution instance of $\bw$ contains
the corresponding substitution instance of $\bu$.

We conclude with the flat semirings occurring in the sequel.  Let
$W$ be a nonempty subset of $X_c^+$.  Denote by $S_c(W)$ the set
consisting of all nonempty divisors of the words in $W$, together
with a new element $0$.  For nonzero $\bu,\bv\in S_c(W)$, put
\[
\bu+\bv=
\begin{cases}
\bu,&\bu=\bv,\\
0,&\bu\neq\bv,
\end{cases}
\qquad
\bu\cdot\bv=
\begin{cases}
\bu\bv,&\bu\bv\text{ divides some word in }W,\\
0,&\text{otherwise}.
\end{cases}
\]
In addition, $0+s=0$ and $0s=s0=0$ for every $s\in S_c(W)$.
Then $S_c(W)$ is a commutative flat ai-semiring, with $0$ as the
greatest element of its natural order and as its multiplicative zero;
see~\cite{jrz}.  If $W=\{\bw\}$, we write $S_c(\bw)$ instead of
$S_c(W)$.

\section{Equational bases and omission criteria}
In this section, we collect the equational bases and omission criteria
needed for the finite basis classification.
\subsection{Auxiliary results}
Let \( n \geq 1 \) be an integer. Define
\[
{\bq}^{(n)} = \prod_{i=1}^{2n+1} x_i,
\]
and
\[
{\bu}^{(n)} = \left(\sum_{i=1}^{2n} x_i x_{i+1}\right)+x_{2n+1} x_{1}.
\]
We denote by $\sigma_n$ the inequality ${\bq}^{(n)} \preceq {\bu}^{(n)}$, and write $\Omega$
for the set of all inequalities $\sigma_n$ with $n \geq 1$.
The inequalities $\sigma_n$ already proved useful in \cite{lry, gmrz, wrz},
where they were employed to establish the nonfinite basis property of some algebras.

For a word $\bp$, let $\ell(\bp)$ denote the length of $\bp$, that is,
the number of variables occurring in $\bp$ counting multiplicities.
For a term $\bt$, $c(\bt)$ denotes the content of $\bt$, that is,
the set of variables occurring in $\bt$.
For an integer $k\geq 1$,
let $L_k({\bt})$ denote the term that is the sum of all words in $\bt$ of length $k$.

Now define a general graph ${\mathbb G}_{\bt}$ with vertex set $\mathsf{V}({\mathbb G}_{\bt})=c(L_2({\bt}))$
and edge set $E({\mathbb G}_{\bt})=\{\{x, y\} \mid xy \in L_2({\bt})\}$.
Note that this graph may contain loops (corresponding to the words of the form $x^2$)
but has no multiple edges.
This graph will play a key role in the analysis that follows.
In particular, for $\mathbf{t} = \mathbf{u}^{(n)}$, the graph $\mathbb{G}_{\mathbf{u}^{(n)}}$ is an odd cycle of length $2n+1$,
with vertices $x_1,\dots,x_{2n+1}$ and edges $\{x_i, x_{i+1}\}$ (indices taken modulo $2n+1$).

Define the odd path closure ${\mathbb G}_{\bt}^{odd}$ of ${\mathbb G}_{\bt}$ by
\[
{\mathbb G}_{\bt}^{odd}=\{xy\mid\text{there is an odd path in ${\mathbb G}_{\bt}$ between $x$ and $y$}\}.
\]

The following result from \cite{lry} gives a sufficient condition for an additively idempotent semiring to be nonfinitely based.

\begin{thm}\label{free02}
Let $S$ be a commutative ai-semiring and $\Sigma$ an equational basis for $S$ that contains an infinite subset of $\Omega$.
If $\mathbf{u}^{(m)}$ is $\mathbf{t}$-free for every inequality $\mathbf{s} \preceq \mathbf{t}$ in $\Sigma \setminus \Omega$ and for every integer $m \geq 1$, then $S$ is nonfinitely based.
\end{thm}

The following characterization of the identities of $S_c(ab)$ is also taken from \cite{lry}.

\begin{lem}\label{lem01}
Let $\bq\preceq \bu$ be a nontrivial inequality such that
$\bu=\bu_1+\bu_2+\cdots+\bu_n$ with $\bu_i, \bq \in X^+_c$ for $1\leq i \leq n$.
Then $\bq\preceq \bu$ is satisfied by $S_c(ab)$ if and only if $\bu$ and $\bq$ satisfy one of the following conditions\up:
\begin{enumerate}[$(\rm i)$]
\item  $\ell(\bu_i)\geq 3$ for some $\bu_i\in \bu$;

\item  $c(L_1(\bu))\cap c(L_2(\bu))\neq \emptyset$;

\item  the graph ${\mathbb G}_{\bu}$ contains an odd cycle;

\item  $\bq\in {\mathbb G}_{\bu}^{odd}$.
\end{enumerate}
\end{lem}

We shall also use the following identity criteria for the three ai-semirings $M_2$, $D_2$ and $T_2$, obtained in \cite{sr}.

\begin{lem}\label{lem:small-tests}
Let $\bu\approx \bu+\bq$ be a nontrivial ai-semiring identity such that
$\bu=\bu_1+\cdots+\bu_n$, where $\bu_i, \bq\in X^+$, $1\leq i \leq n$. Then
\begin{itemize}
\item[$(1)$] $\bq\preceq \bu$ holds in $M_2$ if and only if $c(\bq)\subseteq c(\bu)$.

\item[$(2)$] $\bq\preceq \bu$ holds in $D_2$ if and only if $c(\bu_i)\subseteq c(\bq)$ for some $\bu_i \in \bu$.

\item[$(3)$] $\bq\preceq \bu$ holds in $T_2$ if and only if $\ell(\bu_i)\geq 2$ for some $\bu_i \in \bu$.
\end{itemize}
\end{lem}

\subsection{Equational bases for the small varieties}

\begin{pro}\label{promd}
$\mathsf{V}(M_{2}, D_{2})$ is the ai-semiring variety defined by the following identities
\begin{align}
&x^{2} \approx x; \label{eq:1}\\
&xy \approx yx; \label{eq:2}\\
&xy \preceq x+y; \label{eq:3}\\
&xz \preceq x+yz. \label{eq:4}
\end{align}
\end{pro}
\begin{proof}
It is straightforward to verify that both $M_{2}$ and $D_{2}$ satisfy identities \eqref{eq:1}--\eqref{eq:4}.
To complete the proof, we need to show that every ai-semiring identity satisfied by both $M_{2}$ and $D_{2}$
can be derived from \eqref{eq:1}--\eqref{eq:4}.
Let $\bq \preceq \bu$ be such a nontrivial inequality,
where $\bu=\bu_1+\bu_2+\cdots+\bu_n$ with $\bu_i, \bq\in X^+$ for $1 \leq i \leq n$.
Then $c(\bq)\subseteq c(\bu)$; and there exists $\bu_i \in \bu$ such that $c(\bu_i)\subseteq c(\bq)$.

If $\ell(\bq)=1$, then $c(\bu_i)\subseteq c(\bq)$ and identity \eqref{eq:1} imply $\bu\succeq\bq$ is immediate.
Now suppose $\ell(\bq)\geq 2$. Let $m=\max\{m(x_s,\bq)\mid x_s\in X\}$. Then
\[
\bu \succeq \bu_i+m(\bu_1+\bu_2+\cdots+\bu_n)
\stackrel{\eqref{eq:3}}{\succeq} \bu_i+\bu_1^m \bu_2^m \cdots \bu_n^m
\stackrel{\eqref{eq:2}}{\approx} p\bq+\bu_i
\stackrel{\eqref{eq:4}}{\succeq} \bu_i\bq
\stackrel{\eqref{eq:1},\eqref{eq:2}}{\approx} \bq.
\]
Thus the inequality $\bu \succeq \bq$ is derived.
\end{proof}

\begin{pro}\label{promt}
$\mathsf{V}(M_{2}, T_{2})$ is the ai-semiring variety defined by identities \eqref{eq:2}, together with
\begin{align}
&x \preceq xy; \label{eq:5}\\
&xyz \preceq xy+z. \label{eq:6}
\end{align}
\end{pro}
\begin{proof}
It is straightforward to verify that both $M_{2}$ and $T_{2}$ satisfy identities \eqref{eq:2}, \eqref{eq:5}, and \eqref{eq:6}.
To complete the proof, we need to show that every ai-semiring identity satisfied by both $M_{2}$ and $T_{2}$
can be derived from \eqref{eq:2}, \eqref{eq:5}, and \eqref{eq:6}.
Let $\bq \preceq \bu$ be such a nontrivial inequality,
where $\bu=\bu_1+\bu_2+\cdots+\bu_n$ with $\bu_i, \bq\in X^+$ for $1 \leq i \leq n$.
Then $c(\bq)\subseteq c(\bu)$; and there exists $\bu_i \in \bu$ such that $\ell(\bu_i)\geq 2$.

\textbf{Case 1.} $\ell(\bq)=1$. Then there exists $\bu_j\in \bu$ such that $c(\bq)\subseteq c(\bu_j)$ and $\ell(\bu_j)\geq 2$.
Hence
\[
\bu \succeq \bu_j \stackrel{\eqref{eq:2}}{\approx} \bq p \stackrel{\eqref{eq:5}}{\succeq} \bq.
\]

\textbf{Case 2.} $\ell(\bq)\geq 2$. Take $m=\max\{m(x_s,\bq)\mid x_s\in X\}$. Then
\[
\bu \succeq m(\bu_i+\bu_1+\bu_2+\cdots+\bu_n)
\stackrel{\eqref{eq:6}}{\succeq} \bu_i^m \bu_1^m \bu_2^m \cdots \bu_n^m
\stackrel{\eqref{eq:2}}{\approx} \bq p
\stackrel{\eqref{eq:5}}{\succeq} \bq.
\]
This derives the inequality $\bu \succeq \bq$.
\end{proof}

\begin{pro}\label{promdt}
$\mathsf{V}(M_{2}, D_{2}, T_{2})$ is the ai-semiring variety defined by identities \eqref{eq:2}, \eqref{eq:4}, and \eqref{eq:6}, together with
\begin{align}
&x^{2}y \approx xy; \label{eq:7}\\
&x \preceq x^{2}. \label{eq:8}
\end{align}
\end{pro}
\begin{proof}
It is straightforward to verify that $M_{2}$, $D_{2}$ and $T_{2}$ satisfy identities \eqref{eq:2}, \eqref{eq:4}, and \eqref{eq:6}--\eqref{eq:8}.
To complete the proof, we need to show that every ai-semiring identity satisfied by $M_{2}$, $D_{2}$ and $T_{2}$
can be derived from \eqref{eq:2}, \eqref{eq:4}, and \eqref{eq:6}--\eqref{eq:8}.
Let $\bq \preceq \bu$ be such a nontrivial inequality,
where $\bu=\bu_1+\bu_2+\cdots+\bu_n$ with $\bu_i, \bq\in X^+$ for $1 \leq i \leq n$.
Then $c(\bq)\subseteq c(\bu)$, and there exist $\bu_i, \bu_j \in \bu$ such that $c(\bu_i)\subseteq c(\bq)$ and $\ell(\bu_j)\geq 2$.

\textbf{Case 1.} $\ell(\bq)=1$. Then $c(\bq)=c(\bu_i)$; by identity \eqref{eq:7}, we may assume $\bu_i=\bq^2$.
Consequently,
\[
\bu \succeq \bu_i = \bq^2 \stackrel{\eqref{eq:8}}{\succeq} \bq.
\]

\textbf{Case 2.} $\ell(\bq)\geq 2$. Let $m=\max\{m(x_s,\bq)\mid x_s\in X\}$. Then
\begin{align*}
\bu
&\succeq m(\bu_j+\bu_1+\bu_2+\cdots+\bu_n)+\bu_i\\
&\succeq \bu_j^m\bu_1^m\bu_2^m\cdots\bu_n^m+\bu_i &&(\text{by}~\eqref{eq:6})\\
&\approx \bq p+\bu_i &&(\text{by}~\eqref{eq:2})\\
&\succeq \bu_i\bq &&(\text{by}~\eqref{eq:4})\\
&\approx \bq. &&(\text{by}~\eqref{eq:2},\eqref{eq:7})
\end{align*}
Thus the inequality $\bu \succeq \bq$ is derived.
\end{proof}

\begin{pro}\label{prodt}
$\mathsf{V}(D_{2}, T_{2})$ is the ai-semiring variety defined by identities \eqref{eq:2}, \eqref{eq:7}, and \eqref{eq:8}, together with
\begin{align}
&xt \preceq x+yz. \label{eq:9}
\end{align}
\end{pro}
\begin{proof}
It is straightforward to verify that both $D_{2}$ and $T_{2}$ satisfy identities \eqref{eq:2} and \eqref{eq:7}--\eqref{eq:9}.
To complete the proof, we need to show that every ai-semiring identity satisfied by both $D_{2}$ and $T_{2}$
can be derived from \eqref{eq:2} and \eqref{eq:7}--\eqref{eq:9}.
Let $\bq \preceq \bu$ be such a nontrivial inequality,
where $\bu=\bu_1+\bu_2+\cdots+\bu_n$ with $\bu_i, \bq\in X^+$ for $1 \leq i \leq n$.
Then there exist $\bu_i, \bu_j \in \bu$ such that $c(\bu_i)\subseteq c(\bq)$ and $\ell(\bu_j)\geq 2$.

\textbf{Case 1.} $\ell(\bq)=1$. Since $c(\bu_i)\subseteq c(\bq)$, identity \eqref{eq:7} implies $\bu_i=\bq^2$.
Consequently,
\[
\bu \succeq \bu_i = \bq^2 \stackrel{\eqref{eq:8}}{\succeq} \bq.
\]

\textbf{Case 2.} $\ell(\bq)\geq 2$. Then
\[
\bu \succeq \bu_i+\bu_j \stackrel{\eqref{eq:9}}{\succeq} \bu_i\bq
\stackrel{\eqref{eq:2},\eqref{eq:7}}{\approx} \bq.
\]
Thus the inequality $\bu \succeq \bq$ is derived.
\end{proof}

The following description of
$\mathsf{V}(M_2^0,T_2^0)$ was obtained in
\cite[Proposition~7.4]{yrzs}.

\begin{lem}\label{m0t0}
The variety
$\mathsf{V}(M_2^0,T_2^0)=\mathsf{V}(S_{(4,431)})$
is defined by identities \eqref{eq:2} and
\eqref{eq:6}--\eqref{eq:8}.
\end{lem}

The next equational basis is due to \cite{rlyc}.

\begin{pro}\label{promt0}
The variety $\mathsf{V}(M_2,T_2^0)$ is defined by identities
\eqref{eq:2} and \eqref{eq:6}--\eqref{eq:8}, together with
\begin{align}
xyz \preceq xy+xyzt. \label{eq:10}
\end{align}
\end{pro}

\begin{pro}\label{prom0t}
$\mathsf{V}(M^0_{2}, T_{2})$ is the ai-semiring variety defined by identities \eqref{eq:2}, \eqref{eq:7}, and \eqref{eq:8}, together with
\begin{align}
&zt \preceq xy+z+t. \label{eq:11}
\end{align}
\end{pro}
\begin{proof}
It is straightforward to verify that both $M_2^0$ and $T_2$
satisfy \eqref{eq:2}, \eqref{eq:7}, \eqref{eq:8}, and \eqref{eq:11}.
Conversely, let $\bq\preceq\bu$ be a nontrivial inequality
satisfied by both algebras, where $\bu=\bu_1+\cdots+\bu_n$.
By their equational characterizations, there exist summands
$\bu_{i_1},\ldots,\bu_{i_k}$ and $\bu_j$ such that
\[
c(\bq)=\bigcup_{h=1}^k c(\bu_{i_h}),
\qquad \ell(\bu_j)\geq 2.
\]
Writing $\bu_j=\mathbf{r}\mathbf{s}$, we obtain from
\eqref{eq:11} that
\[
\bu^2\preceq\mathbf{r}\mathbf{s}+\bu+\bu\approx\bu,
\]
and hence $\bu^m\preceq\bu$ for every $m\geq1$.
Therefore,
\[
\bu
\succeq\bu^{2k}
\succeq(\bu_{i_1}\cdots\bu_{i_k})^2
\stackrel{\eqref{eq:2},\eqref{eq:7}}{\approx}\bq^2
\stackrel{\eqref{eq:8}}{\succeq}\bq.
\]
Thus $\bq\preceq\bu$ follows from
\eqref{eq:2}, \eqref{eq:7}, \eqref{eq:8}, and \eqref{eq:11}.
Thus the inequality $\bu \succeq \bq$ is derived.
\end{proof}

\subsection{The interval below $\mathsf{V}(M_2^0,T_2^0)$}

\begin{pro}\label{prosmallexclusions}
Let $\mathcal V\leq\mathsf{V}(M_2^0,T_2^0)$.
Then the following statements hold.
\begin{enumerate}
\item
$T_2^0\notin\mathcal V$ if and only if $\mathcal V$ satisfies
\begin{equation}\label{eq:omitT20}
x+y^2\succeq x^2.
\end{equation}
Equivalently, $\mathcal V\leq\mathsf{V}(M_2^0,T_2)$.

\item
$M_2^0\notin\mathcal V$ if and only if $\mathcal V$ satisfies
\begin{equation}\label{eq:omitM20}
x^2+yz\succeq xy.
\end{equation}
Equivalently, $\mathcal V\leq\mathsf{V}(M_2,T_2^0)$.

\item
$D_2\notin\mathcal V$ if and only if $\mathcal V$ satisfies
\begin{equation}\label{eq:omitD2-small}
xy\approx x^2+y^2.
\end{equation}
Equivalently, $\mathcal V\leq\mathsf{V}(M_2,T_2)$.
\end{enumerate}
\end{pro}

\begin{proof}
By \eqref{eq:2} and \eqref{eq:6}--\eqref{eq:8},
every member of $\mathsf{V}(M_2^0,T_2^0)$ satisfies
\[
(xy)^2\approx xy,\qquad
(x+y)^2\approx x^2+y^2\succeq xy.
\]
It is straightforward to verify that the three varieties
in (1)--(3) satisfy the corresponding identities, whereas
$T_2^0$, $M_2^0$, and $D_2$, respectively, fail them.
We prove the remaining implications.

$(1)$ Suppose that $T_2^0\notin\mathcal V$.
If $\mathcal V$ fails \eqref{eq:omitT20}, there exist
$S\in\mathcal V$ and $a,b\in S$ such that
$a^2\not\leq a+b^2$.
Using \eqref{eq:2} and \eqref{eq:6}--\eqref{eq:8}, we obtain
\[
b^2+ab\leq a+b^2< a^2+b^2.
\]
The first inequality is also strict, since
\[
(b^2+ab)^2=b^2+ab,\qquad
(a+b^2)^2=a^2+b^2.
\]
A direct verification shows that
\[
\{b^2+ab,\ a+b^2,\ a^2+b^2\}
\]
is a subalgebra of $S$ isomorphic to $T_2^0$, a contradiction.
Thus $\mathcal V$ satisfies \eqref{eq:omitT20}.

Moreover,
\[
xy+z+t
\approx (xy)^2+z+t
\stackrel{\eqref{eq:omitT20}}{\succeq}
(z+t)^2
\succeq zt.
\]
Hence $\mathcal V$ satisfies \eqref{eq:11}.
By Proposition~\ref{prom0t},
$\mathcal V\leq\mathsf{V}(M_2^0,T_2)$.

$(2)$ Suppose that $M_2^0\notin\mathcal V$.
If $\mathcal V$ fails \eqref{eq:omitM20}, there exist
$S\in\mathcal V$ and $a,b,c\in S$ such that
$ab\not\leq a^2+bc$.
Then
\[
abc\leq a^2+bc< a^2+bc+ab.
\]
The first inequality is strict: if $abc=a^2+bc$,
multiplication by $a^2+bc+ab$ gives
$abc=a^2+bc+ab$, a contradiction.
Using \eqref{eq:2} and \eqref{eq:6}--\eqref{eq:8},
one verifies that
\[
\{abc,\ a^2+bc,\ a^2+bc+ab\}
\]
is a subalgebra of $S$ isomorphic to $M_2^0$.
This is again a contradiction.
Thus $\mathcal V$ satisfies \eqref{eq:omitM20}.

Consequently,
\[
xy+xyzt
\approx (xy)^2+z(xyt)
\stackrel{\eqref{eq:omitM20}}{\succeq}
(xy)z
\approx xyz.
\]
Hence $\mathcal V$ satisfies \eqref{eq:10}.
By Proposition~\ref{promt0},
$\mathcal V\leq\mathsf{V}(M_2,T_2^0)$.

$(3)$ Suppose that $D_2\notin\mathcal V$.
If $\mathcal V$ fails \eqref{eq:omitD2-small}, there exist
$S\in\mathcal V$ and $a,b\in S$ such that
\[
ab< a^2+b^2.
\]
Both elements are multiplicatively idempotent, and
\[
ab(a^2+b^2)=ab.
\]
Thus $\{ab,a^2+b^2\}$ is a subalgebra of $S$
isomorphic to $D_2$, a contradiction.
Therefore $\mathcal V$ satisfies \eqref{eq:omitD2-small}.
Now
\[
xy
\stackrel{\eqref{eq:omitD2-small}}{\approx}
x^2+y^2
\stackrel{\eqref{eq:8}}{\succeq}
x.
\]
By Proposition~\ref{promt},
$\mathcal V\leq\mathsf{V}(M_2,T_2)$.
\end{proof}

\begin{pro}\label{pro:five-varieties}
The interval
\[
[\mathsf{V}(M_2,T_2),\mathsf{V}(M_2^0,T_2^0)]
\]
consists precisely of the following five varieties:
\[
\begin{gathered}
\mathsf{V}(M_2,T_2),\qquad
\mathsf{V}(M_2,T_2,D_2),\\
\mathsf{V}(M_2^0,T_2),\qquad
\mathsf{V}(M_2,T_2^0),\qquad
\mathsf{V}(M_2^0,T_2^0).
\end{gathered}
\]
All five varieties are finitely based.
\end{pro}

\begin{proof}
Let
\[
\mathsf{V}(M_2,T_2)\leq\mathcal V
\leq\mathsf{V}(M_2^0,T_2^0).
\]
We distinguish four cases.

\textbf{Case 1.} $M_2^0,T_2^0\in\mathcal V$.
Then $\mathcal V=\mathsf{V}(M_2^0,T_2^0)$.

\textbf{Case 2.} $M_2^0\in\mathcal V$ and
$T_2^0\notin\mathcal V$.
By Proposition~\ref{prosmallexclusions},
\[
\mathsf{V}(M_2^0,T_2)\leq\mathcal V
\leq\mathsf{V}(M_2^0,T_2).
\]
Hence $\mathcal V=\mathsf{V}(M_2^0,T_2)$.

\textbf{Case 3.} $M_2^0\notin\mathcal V$ and
$T_2^0\in\mathcal V$.
Similarly, $\mathcal V=\mathsf{V}(M_2,T_2^0)$.

\textbf{Case 4.} $M_2^0,T_2^0\notin\mathcal V$.
By Proposition~\ref{prosmallexclusions},
$\mathcal V$ satisfies \eqref{eq:omitT20}
and \eqref{eq:omitM20}. Therefore
\[
x+yz
\stackrel{\eqref{eq:omitT20},\eqref{eq:7}}{\succeq}
x^2+yz
\stackrel{\eqref{eq:omitM20}}{\succeq}
xz.
\]
By Proposition~\ref{promdt},
\[
\mathcal V\leq\mathsf{V}(M_2,T_2,D_2).
\]
If $D_2\in\mathcal V$, then
$\mathcal V=\mathsf{V}(M_2,T_2,D_2)$.
Otherwise, Proposition~\ref{prosmallexclusions} gives
$\mathcal V\leq\mathsf{V}(M_2,T_2)$, and hence
$\mathcal V=\mathsf{V}(M_2,T_2)$.

All five displayed varieties belong to the interval,
and their distinctness follows from
Proposition~\ref{prosmallexclusions}.
Their finite bases were given in
Propositions~\ref{promt}, \ref{promdt}, \ref{prom0t},
and \ref{promt0}, and Lemma~\ref{m0t0}.
\end{proof}

\begin{pro}\label{pro:three-varieties}
The interval
\[
[\mathsf{V}(T_2),\mathsf{V}(T_2^0)]
\]
consists precisely of
\[
\mathsf{V}(T_2),\qquad
\mathsf{V}(T_2,D_2),\qquad
\mathsf{V}(T_2^0).
\]
In particular, all three varieties are finitely based.
\end{pro}
\begin{proof}
Let
$\mathsf{V}(T_2)\leq\mathcal V\leq\mathsf{V}(T_2^0)$.
If $T_2^0\in\mathcal V$, then
$\mathcal V=\mathsf{V}(T_2^0)$. Suppose that
$T_2^0\notin\mathcal V$. By
Proposition~\ref{prosmallexclusions}, $\mathcal V$ satisfies
\eqref{eq:omitT20}. Moreover, $T_2^0$ satisfies
$xy\succeq xyz$. Hence
\[
x+yz
\stackrel{\eqref{eq:omitT20}}{\succeq}x^2
\succeq x^2t
\stackrel{\eqref{eq:7}}{\approx}xt.
\]
Thus $\mathcal V$ satisfies \eqref{eq:9}, and
Proposition~\ref{prodt} gives
$\mathcal V\leq\mathsf{V}(T_2,D_2)$.
If $D_2\in\mathcal V$, equality follows. Otherwise,
Proposition~\ref{pro25032103} gives \eqref{eq:25}, and
\[
xy\approx x^2y^2
\stackrel{\eqref{eq:25}}{\succeq}x^2
\stackrel{\eqref{eq:8}}{\succeq}x.
\]
Together with \eqref{eq:2} and \eqref{eq:6}, this is the basis
of $\mathsf{V}(T_2)$ in Proposition~\ref{promt}; hence
$\mathcal V=\mathsf{V}(T_2)$.
\end{proof}

\subsection{The varieties associated with $S_c(ab)$}

\begin{pro}\label{proabm}
$\mathsf{V}(S_c(ab), M_{2})$ is the ai-semiring variety defined by identities \eqref{eq:2} and \eqref{eq:8}, together with
\begin{align}
&x^2 \approx x^3; \label{eq:12}\\
&xyzt \preceq xyz+t; \label{eq:13}\\
&t \preceq xyz+ts; \label{eq:14}\\
&xyz \preceq x+xy+z; \label{eq:15}\\
&xt \preceq xy+yz+zt; \label{eq:16}\\
&a \preceq x+xy+ab; \label{eq:17}\\
&y \preceq x^2+yz; \label{eq:18}\\
&xy \preceq x^2+y. \label{eq:19}
\end{align}
\end{pro}
\begin{proof}
It is straightforward to verify that both $S_c(ab)$ and $M_2$
satisfy \eqref{eq:2}, \eqref{eq:8}, and \eqref{eq:12}--\eqref{eq:19}.
Conversely, let $\bu\succeq\bq$ be a nontrivial inequality
satisfied by both algebras, where $\bu=\bu_1+\cdots+\bu_n$
and $\bu_i,\bq\in X_c^+$.
Since it holds in $M_2$, we have $c(\bq)\subseteq c(\bu)$.
By Lemma~\ref{lem01}, it suffices to consider four cases.

\textbf{Case 1.} $\ell(\bu_j)\geq3$ for some $\bu_j\in\bu$.
By \eqref{eq:14}, $\bu\succeq x$ for every $x\in c(\bu)$.
Hence
\[
\bu
\succeq \bu_j+\sum_{x\in c(\bq)}x
\stackrel{\eqref{eq:13}}{\succeq}
\bu_j+\bu_j\bq
\stackrel{\eqref{eq:14}}{\succeq}
\bq.
\]

\textbf{Case 2.}
$c(L_1(\bu))\cap c(L_2(\bu))\neq\emptyset$.
Then $x,xy\in\bu$ for some variables $x,y$, and
\[
\bu\succeq x+xy
\stackrel{\eqref{eq:15}}{\succeq}x^2y.
\]
Adjoining $x^2y$ reduces the proof to Case~1.

\textbf{Case 3.}
The graph $\mathbb G_{\bu}$ contains an odd cycle.
Repeated applications of \eqref{eq:16} yield
\[
\bu\succeq x^2
\stackrel{\eqref{eq:12}}{\approx}x^3.
\]
Adjoining $x^3$ reduces the proof to Case~1.

\textbf{Case 4.} $\bq\in\mathbb G_{\bu}^{odd}$.
Then $\bq=xy$, where $x,y$ are joined by an odd path
in $\mathbb G_{\bu}$.
Repeated applications of \eqref{eq:16} along this path
give $\bu\succeq xy=\bq$.

Thus $\bu\succeq\bq$ follows from the proposed identities.
\end{proof}

\begin{pro}\label{proab0m0}
$\mathsf{V}(S_c(ab)^0, M^0_{2})$ is the commutative ai-semiring variety defined by identities \eqref{eq:2}, \eqref{eq:8}, \eqref{eq:12}, and \eqref{eq:13},
 $\sigma_n$ for all $n \geq 1$, and
 \begin{align}
&xy^2 \preceq x+xy. \label{eq:20}
\end{align}
\end{pro}
\begin{proof}
It is straightforward to verify that both $S_c(ab)^0$ and
$M_2^0$ satisfy \eqref{eq:2}, \eqref{eq:8}, \eqref{eq:12}, \eqref{eq:13}, and \eqref{eq:20} and $\sigma_n$
for all $n\geq1$.
Conversely, let $\bq\preceq\bu$ be a nontrivial inequality
satisfied by both algebras, where
$\bu=\bu_1+\cdots+\bu_n$ and $\bu_i,\bq\in X_c^+$.
After reindexing, write
\[
D_{\bq}(\bu)=\bu_1+\cdots+\bu_k,
\]
where these are precisely the summands whose contents
are contained in $c(\bq)$.
By the equational characterizations of the zero extensions,
\[
S_c(ab)\models\bq\preceq D_{\bq}(\bu),
\qquad
c(\bq)=\bigcup_{i=1}^{k}c(\bu_i).
\]
By Lemma~\ref{lem01}, it suffices to consider four cases.

\textbf{Case 1.}
$\ell(\bu_j)\geq3$ for some $\bu_j\in D_{\bq}(\bu)$.
Keeping one occurrence of $\bu_j$ and repeatedly applying
\eqref{eq:13}, we obtain
\[
\bu
\succeq D_{\bq}(\bu)
\stackrel{\eqref{eq:13}}{\succeq}
(\bu_1\cdots\bu_k)^2
\stackrel{\eqref{eq:2},\eqref{eq:12}}{\approx}
\bq^2
\stackrel{\eqref{eq:8}}{\succeq}
\bq.
\]

\textbf{Case 2.}
$c(L_1(D_{\bq}(\bu)))\cap c(L_2(D_{\bq}(\bu)))
\neq\emptyset$.
Then $x$ and $xy$ are summands of $D_{\bq}(\bu)$
for some variables $x,y$.
By \eqref{eq:20},
\[
xy^2\preceq x+xy\preceq D_{\bq}(\bu).
\]
Adjoining the summand $xy^2$, we obtain the result
as in Case~1.

\textbf{Case 3.}
The graph $\mathbb G_{D_{\bq}(\bu)}$ contains an odd cycle.
If the cycle is a loop at $x$, then
$x^3\approx x^2\preceq D_{\bq}(\bu)$ by \eqref{eq:12}.
Otherwise, the corresponding identity $\sigma_m$ gives
\[
x_1x_2\cdots x_{2m+1}
\preceq D_{\bq}(\bu).
\]
In either case, adjoining this long summand reduces
the proof to Case~1.

\textbf{Case 4.}
$\bq\in\mathbb G_{D_{\bq}(\bu)}^{odd}$.
Since Case~3 is already excluded, $\bq=xy$ for distinct
variables $x,y$, joined by an odd path in
$\mathbb G_{D_{\bq}(\bu)}$.
Every vertex of this graph belongs to
$c(\bq)=\{x,y\}$, so $xy$ must be a summand of
$D_{\bq}(\bu)$.
Thus $\bq\preceq\bu$ is trivial.

Hence every such inequality follows from the proposed
identities.
\end{proof}

\begin{lem}\label{lem:S70-identities}
The semiring $S_7^0$ satisfies identities \eqref{eq:2},
\eqref{eq:8}, and \eqref{eq:12}, the inequalities $\sigma_n$
for all $n\geq1$, and the following identities:
\begin{align}
x^2y^2&\approx x^2y+xy^2; \label{eq:21}\\
x+xy&\approx x+xy^2; \label{eq:22}\\
x^2+y&\succeq x^2y^2. \label{eq:23}
\end{align}
\end{lem}
\begin{proof}
The identities \eqref{eq:2}, \eqref{eq:8},
\eqref{eq:12}, and \eqref{eq:21}--\eqref{eq:23} follow by a
direct verification in $S_7^0$. To verify $\sigma_n$, let
$\varphi:P_f(X_c^+)\to S_7^0$ be a semiring homomorphism. If
$\varphi(x_i)=0$ for some $i$, then
$\varphi(\bq^{(n)})=0\leq\varphi(\bu^{(n)})$. Otherwise the
calculation takes place in the subalgebra $S_7$, where
$\sigma_n$ holds. Hence $S_7^0$ satisfies $\sigma_n$ for every
$n\geq1$.
\end{proof}

\begin{pro}\label{proabbasis}
$\mathsf{V}(S_c(ab))$ is the commutative ai-semiring
variety defined by identities \eqref{eq:12},
\eqref{eq:16}, and
\begin{align}
x^2 &\approx x+xy; \label{eq:28}\\
t &\preceq xyz. \label{eq:29}
\end{align}
\end{pro}

\begin{pro}\label{proab0basis}
$\mathsf{V}(S_c(ab)^0)$ is the commutative ai-semiring
variety defined by identities \eqref{eq:8},
\eqref{eq:12}, $\sigma_n$ for all $n\geq1$, and
\begin{align}
xyzt &\preceq xyz; \label{eq:30}\\
(xy)^2 &\preceq x+xy. \label{eq:31}
\end{align}
\end{pro}

\begin{pro}\label{proabrelative}
The variety $\mathsf{V}(S_c(ab))$ is the subvariety of
$\mathsf{V}(S_c(ab)^0)$ defined by \eqref{eq:16}.
\end{pro}

\begin{proof}
It is straightforward to verify that $S_c(ab)$ satisfies
\eqref{eq:16}.
Conversely, suppose that $S\in\mathsf{V}(S_c(ab)^0)$
satisfies \eqref{eq:16}.
Then
\[
y^2
\stackrel{\eqref{eq:12},\eqref{eq:30}}{\approx}
y^2+xy^2
\stackrel{\eqref{eq:8}}{\succeq}
y^2+xy
\stackrel{\eqref{eq:16}}{\succeq}
x^2.
\]
By symmetry, $S$ satisfies $x^2\approx y^2$.
Consequently,
\[
x+xy
\stackrel{\eqref{eq:31}}{\succeq}
(xy)^2
\approx x^2
\succeq x+xy,
\]
where the last inequality follows from \eqref{eq:8}
and $x^2\approx y^2$.
Moreover,
\[
xyz
\stackrel{\eqref{eq:30}}{\succeq}
(xyz)^2
\approx t^2
\stackrel{\eqref{eq:8}}{\succeq}
t.
\]
Thus $S$ satisfies \eqref{eq:28} and \eqref{eq:29}.
By Proposition~\ref{proabbasis},
$S\in\mathsf{V}(S_c(ab))$, as required.
\end{proof}

\begin{pro}\label{proabrelative2}
The variety $\mathsf{V}(S_c(ab), M_2)$ is the subvariety of
$\mathsf{V}(S_c(ab)^0, M^0_2)$ defined by \eqref{eq:16}.
\end{pro}
\begin{proof}
It is straightforward to verify that both $S_c(ab)$
and $M_2$ satisfy \eqref{eq:16}.
Conversely, suppose that
$S\in\mathsf{V}(S_c(ab)^0,M_2^0)$ satisfies
\eqref{eq:16}.
By \eqref{eq:8}, \eqref{eq:12} and \eqref{eq:13},
\[
(xyz)^2\approx xyz,
\qquad
xy^2\approx x^2y^2.
\]
Substituting
$(x,y,z,t)=(x,x^2y^2,x^2y^2,x)$
in \eqref{eq:16} gives $x^2y^2\succeq x^2$.
By symmetry and \eqref{eq:13},
\[
x^2+y^2\succeq x^2y^2\succeq x^2+y^2,
\]
and hence $x^2y^2\approx x^2+y^2$.
Moreover,
\[
x+y^2
\stackrel{\eqref{eq:12},\eqref{eq:13}}{\succeq}
xy^2
\approx x^2y^2
\succeq x^2.
\]
Consequently,
\[
\begin{aligned}
x^2+yz
&\succeq (yz)^2
 \approx y^2+z^2
 \succeq y,\\
x^2+y
&\succeq x^2+y^2
 \approx (xy)^2
 \succeq xy.
\end{aligned}
\]
These are \eqref{eq:18} and \eqref{eq:19}.
Using them together with \eqref{eq:20}, we obtain
\[
\begin{aligned}
xyz+ts
&\approx (xyz)^2+ts
 \succeq t,\\
x+xy+z
&\succeq (xy)^2+z
 \succeq xyz,\\
x+xy+ab
&\succeq (xy)^2+ab
 \succeq a.
\end{aligned}
\]
Thus $S$ also satisfies \eqref{eq:14},
\eqref{eq:15} and \eqref{eq:17}.
By Proposition~\ref{proabm},
$S\in\mathsf{V}(S_c(ab),M_2)$, as required.
\end{proof}

\begin{pro}\label{prop:strictcontains}
Let $\mathcal V$ belong to one of the intervals
\[
[\mathsf{V}(S_c(ab)),\mathsf{V}(S_c(ab)^0)],
\qquad
[\mathsf{V}(S_c(ab),M_2),
 \mathsf{V}(S_c(ab)^0,M_2^0)].
\]
If $\mathcal V$ is not the least member of that interval,
then $SR_6\in\mathcal V$.
\end{pro}
\begin{proof}
By the relative descriptions of the least members of the
two intervals, $\mathcal V$ does not satisfy \eqref{eq:16}.
Suppose, for a contradiction, that $SR_6\notin\mathcal V$.
Then there exists a nontrivial inequality
$\bu\succeq\bq$ that holds in $\mathcal V$ but fails in $SR_6$,
where $\bu=\bu_1+\cdots+\bu_n$ and $\bu_i,\bq\in X_c^+$.
Since $S_c(ab)\in\mathcal V$, this inequality holds in
$S_c(ab)$. Thus one of the four conditions in
Lemma~\ref{lem01} holds.

Recall that $SR_6$ satisfies \eqref{eq:12},
\eqref{eq:28}, \eqref{eq:29}, and $\sigma_n$ for all $n\geq1$.
If condition~$(i)$ holds, then, in $SR_6$,
\[
\bu\succeq\bu_i
\stackrel{\eqref{eq:29}}{\succeq}\bq
\]
for a summand $\bu_i$ of length at least three.
If condition~$(ii)$ holds, then
\[
\bu\succeq x+xy
\stackrel{\eqref{eq:28}}{\approx}x^2
\stackrel{\eqref{eq:12}}{\approx}x^3
\stackrel{\eqref{eq:29}}{\succeq}\bq.
\]
If condition~$(iii)$ holds, the corresponding $\sigma_m$
gives
\[
\bu\succeq x_1x_2\cdots x_{2m+1}
\stackrel{\eqref{eq:29}}{\succeq}\bq.
\]
For a loop, use $x^2\approx x^3$ instead.
Each possibility contradicts the choice of
$\bu\succeq\bq$.

Therefore, condition~$(iv)$ holds and conditions
$(i)$--$(iii)$ fail.
In particular,
\[
\bu=L_1(\bu)+L_2(\bu),\qquad
c(L_1(\bu))\cap c(L_2(\bu))=\emptyset,
\]
and $\mathbb G_{\bu}$ is bipartite.
Moreover, $\bq=xy$ for distinct variables $x,y$
joined by an odd path in $\mathbb G_{\bu}$.
Let $(A,B)$ be a bipartition with $x\in A$ and $y\in B$.
Since the inequality is nontrivial, $xy$ is not a summand
of $\bu$, so this path has length at least three.

Choose fresh variables $y_1,y_2,y_3,y_4$ and define
a substitution $\psi$ by
\[
\psi(z)=
\begin{cases}
y_1, & z=x,\\
y_4, & z=y,\\
y_3, & z\in A\setminus\{x\},\\
y_2, & z\in B\setminus\{y\},\\
y_2y_3, & \text{otherwise}.
\end{cases}
\]
Then
\[
\psi(\bu)\approx y_1y_2+y_2y_3+y_3y_4,
\qquad
\psi(\bq)=y_1y_4.
\]
Since $\mathcal V$ satisfies $\bu\succeq\bq$,
it also satisfies
\[
y_1y_2+y_2y_3+y_3y_4\succeq y_1y_4.
\]
This is \eqref{eq:16}, contradicting the choice of
$\mathcal V$. Hence $SR_6\in\mathcal V$.
\end{proof}

\subsection{Nonfinitely based intervals}

\begin{thm}\label{thmrelativeNFB}
Let $\mathcal A\leq\mathcal B$ be varieties of the same type.
If $\mathcal A$ is nonfinitely based but finitely based
within $\mathcal B$, then every variety in the interval
$[\mathcal A,\mathcal B]$ is nonfinitely based.
\end{thm}

\begin{proof}
Let $\Sigma$ be a finite set of identities defining
$\mathcal A$ within $\mathcal B$, and let
$\mathcal V\in[\mathcal A,\mathcal B]$.
Since $\mathcal A\leq\mathcal V\leq\mathcal B$,
the same set $\Sigma$ defines $\mathcal A$ within
$\mathcal V$.
If $\mathcal V$ had a finite equational basis $\Delta$,
then $\Delta\cup\Sigma$ would be a finite equational
basis of $\mathcal A$, a contradiction.
Hence $\mathcal V$ is nonfinitely based.
\end{proof}

\begin{pro}\label{proSR}
$\mathsf{V}(SR_6)$ is the commutative ai-semiring variety defined by the identities
\eqref{eq:12}, \eqref{eq:28}, \eqref{eq:29} and $\sigma_n$ for all $n \geq 1$.
\end{pro}

\begin{pro}\label{proSRrelative}
The variety $\mathsf{V}(SR_6)$ is the subvariety of
$\mathsf{V}(S_7^0)$ defined by \eqref{eq:29}.
Consequently, every variety in the interval
\[
[\mathsf{V}(SR_6),\mathsf{V}(S_7^0)]
\]
is nonfinitely based.
\end{pro}

\begin{proof}
We know that
$SR_6\in\mathsf{V}(S_c(ab)^0)\leq\mathsf{V}(S_7^0)$
and that $SR_6$ satisfies \eqref{eq:29}.
Conversely, suppose that $S\in\mathsf{V}(S_7^0)$
satisfies \eqref{eq:29}.
Then
\[
x^2
\stackrel{\eqref{eq:12}}{\approx}
x^3
\stackrel{\eqref{eq:29}}{\succeq}
t.
\]
Thus every square is the greatest element of $S$.
Using \eqref{eq:22}, we obtain
\[
x+xy
\stackrel{\eqref{eq:22}}{\approx}
x+xy^2
\stackrel{\eqref{eq:29}}{\succeq}
x^2
\succeq x+xy.
\]
Hence $S$ satisfies \eqref{eq:28}.
Since \eqref{eq:2}, \eqref{eq:12},
\eqref{eq:28}, \eqref{eq:29}, and $\sigma_n$
for all $n\geq1$ form an equational basis of $SR_6$,
we obtain $S\in\mathsf{V}(SR_6)$.

Therefore, $\mathsf{V}(SR_6)$ is finitely based within
$\mathsf{V}(S_7^0)$.
Since $\mathsf{V}(SR_6)$ is nonfinitely based,
Theorem~\ref{thmrelativeNFB} completes the proof.
\end{proof}

\begin{cor}\label{cor:twointervals}
Each of the intervals
\[
[\mathsf{V}(S_c(ab)),\mathsf{V}(S_c(ab)^0)],
\qquad
[\mathsf{V}(M_2,S_c(ab)),
 \mathsf{V}(M_2^0,S_c(ab)^0)]
\]
contains exactly one finitely based variety, namely its
least member.
\end{cor}

\begin{proof}
The varieties $\mathsf{V}(S_c(ab))$ and
$\mathsf{V}(M_2,S_c(ab))$ are finitely based by the
equational bases established above.
Let $\mathcal V$ be any other member of either interval.
By Proposition~\ref{prop:strictcontains},
$SR_6\in\mathcal V$. Hence
\[
\mathsf{V}(SR_6)\leq\mathcal V
\leq\mathsf{V}(S_7^0).
\]
Proposition~\ref{proSRrelative} now implies that
$\mathcal V$ is nonfinitely based.
\end{proof}

\subsection{Omission criteria in $\mathsf{V}(S_7^0)$}

\begin{pro}\label{pro25032101}
Let $\mathcal{V}$ be a subvariety of $\mathsf{V}(S_7^0)$.
Then $\mathcal{V}$ does not contain $S_c(a)$
if and only if $\mathcal{V}$ satisfies the identity
\eqref{eq:1}.
\end{pro}
\begin{proof}
Suppose that $\mathcal{V}$ satisfies~\eqref{eq:1}.
It is easy to see that $S_c(a)$ does not satisfy~\eqref{eq:1}
and so $S_c(a)$ is not contained in $\mathcal{V}$.
Conversely, assume that $\mathcal{V}$ does not satisfy~\eqref{eq:1}.
Then there exists a semiring $S$ in $\mathcal{V}$ such that $a^2\neq a$ for some $a\in S$.
Since the identities~\eqref{eq:12} and~\eqref{eq:8} are satisfied by $S_7^0$,
it follows that $\{a^2, a\}$ is isomorphic to $S_c(a)$.
Thus $\mathcal{V}$ contains $S_c(a)$ as required.
\end{proof}

We shall use $\mathbf{I}$ to denote the subvariety of $\mathsf{V}(S_7^0)$ defined by the identity~\eqref{eq:1}.
\begin{pro}\label{pro25032110}
The variety $\mathbf{I}$ has $5$ subvarieties\up:
$\mathsf{V}(M_2^0)$, $\mathsf{V}(M_2, D_2)$, $\mathsf{V}(M_2)$, $\mathsf{V}(D_2)$ and the trivial variety $\mathbf{T}$.
\end{pro}
\begin{proof}
One can show that $\mathbf{I}$ coincides with the ai-semiring variety defined by identities \eqref{eq:1} and \eqref{eq:2},
which is an equational basis of $\mathsf{V}(M_2^0)$. Indeed, we have
\[
\mathsf{V}(S_7^0) \wedge [\eqref{eq:1}]
=\mathsf{V}(S_7^0) \wedge [\eqref{eq:1}, \eqref{eq:2}]=\mathsf{V}(S_7^0) \wedge \mathsf{V}(M_2^0)=\mathsf{V}(M_2^0).
\]
So the required result is true.
\end{proof}

\begin{remark}
$\mathsf{V}(M_2^0)$ is the subvariety of $\mathsf{V}(S_7^0)$ defined by the identity \eqref{eq:1}.
\end{remark}

\begin{pro}\label{pro25032102}
Let $\mathcal{V}$ be a subvariety of $\mathsf{V}(S_7^0)$.
Then $\mathcal{V}$ does not contain $M_2$
if and only if $\mathcal{V}$ satisfies the identity
\begin{equation}\label{eq:24}
x^2\approx x^2+x^2y^2.
\end{equation}
\end{pro}
\begin{proof}
Suppose that $\mathcal{V}$ satisfies the identity (\ref{eq:24}). Since this identity
does not hold in $M_2$, it follows immediately that $\mathcal{V}$ does not contain $M_2$.
Conversely, assume that~$\mathcal{V}$ does not satisfy the identity (\ref{eq:24}).
Then there exists $S$ in $\mathcal{V}$ such that $a^2\neq a^2+a^2b^2$ for some $a, b \in S$.
By using the identities~\eqref{eq:12} and\eqref{eq:2},
it is easy to verify that $\{a^2, a^2+a^2b^2\}$ is isomorphic to $M_2$.
Hence $\mathcal{V}$ contains $M_2$ as required.
\end{proof}

\begin{pro}\label{pro25032103}
Let $\mathcal{V}$ be a subvariety of $\mathsf{V}(S_7^0)$.
Then $\mathcal{V}$ does not contain $D_2$ if and only if
$\mathcal{V}$ satisfies the identity
\begin{equation}\label{eq:25}
x^2y^2 \approx x^2y^2+x^2.
\end{equation}
\end{pro}
\begin{proof}
Assume that $\mathcal{V}$ satisfies the identity (\ref{eq:25}). Since this identity
does not hold in $D_2$, it follows immediately that $\mathcal{V}$ does not contain $D_2$.
Conversely, assume that~$\mathcal{V}$ does not satisfy the identity (\ref{eq:25}).
Then there exists $S$ in $\mathcal{V}$ such that $a^2b^2 \neq a^2b^2+a^2$ for some $a, b \in S$.
By using the identities~\eqref{eq:12} and~\eqref{eq:2},
it is easy to verify that $\{a^2b^2, a^2b^2+a^2\}$ is isomorphic to $D_2$.
So $\mathcal{V}$ contains $D_2$.
\end{proof}

\begin{pro}\label{proab}
Let $\mathcal V\leq\mathsf{V}(S_7^0)$.
Then $S_c(ab)\notin\mathcal V$ if and only if
$\mathcal V$ satisfies
\begin{equation}\label{eq:26}
(xy)^2\approx xy.
\end{equation}
\end{pro}

\begin{proof}
Suppose that $\mathcal V$ satisfies \eqref{eq:26}.
Since $S_c(ab)$ fails this identity,
$S_c(ab)\notin\mathcal V$.
Conversely, suppose that $\mathcal V$ fails
\eqref{eq:26}.
By \eqref{eq:8}, there exist $S\in\mathcal V$
and $a,b\in S$ such that $ab<a^2b^2$.
Let $B=\langle a,b\rangle$.
By \eqref{eq:12}, \eqref{eq:2} and distributivity,
every element of $B$ is a nonempty sum of expressions from
\[
\{a,b,ab,a^2,b^2,a^2b,ab^2,a^2b^2\}.
\]

\textbf{Case 1.}
$ab\neq a^2b$ and $ab\neq ab^2$.
Using \eqref{eq:2}, \eqref{eq:8}, \eqref{eq:12}, and \eqref{eq:21}--\eqref{eq:23}, one checks that a sum
of the listed expressions can equal $a$, $b$ or $ab$
only when all its summands are, respectively,
$a$, $b$ or $ab$.
Consider the subsets
\[
R_1=\{a\},\qquad
R_2=\{b\},\qquad
R_3=\{ab\},\qquad
R_4=B\setminus\{a,b,ab\}.
\]
These sets are pairwise disjoint and nonempty;
in particular, $a^2b^2\in R_4$.
The preceding observation gives
\[
R_4+B\subseteq R_4,\qquad
R_4B\subseteq R_4.
\]
Moreover, sums of distinct singleton classes lie in $R_4$,
and their products satisfy
\[
R_1R_2=R_2R_1=R_3,
\]
whereas
\[
R_iR_j\subseteq R_4
\quad\text{for }
1\leq i,j\leq4,\quad
(i,j)\notin\{(1,2),(2,1)\}.
\]
Thus the equivalence relation $\rho$ with classes
$R_1,R_2,R_3,R_4$ is compatible with both operations.
Hence $\rho$ is a semiring congruence and
$B/\rho\cong S_c(ab)$.
Consequently, $S_c(ab)\in\mathcal V$.

\textbf{Case 2.}
$ab=a^2b$ or $ab=ab^2$.
By symmetry, assume that $ab=a^2b$.
Then $ab^2=a^2b^2$.
Set
\[
R_1=\{a,a^2\},\qquad
R_2=\{b,ab,b+ab\},\qquad
R_3=B\setminus(R_1\cup R_2).
\]
A direct verification using \eqref{eq:2}, \eqref{eq:8}, \eqref{eq:12}, and \eqref{eq:21}--\eqref{eq:23}
and $ab<a^2b^2$ shows that these sets are pairwise
disjoint and nonempty, and that
\[
\begin{aligned}
R_i+R_i&\subseteq R_i
&& (1\leq i\leq3),\\
R_i+R_j&\subseteq R_3
&& (1\leq i,j\leq3,\ i\neq j).
\end{aligned}
\]
Their products satisfy
\[
\begin{gathered}
R_1R_1\subseteq R_1,\qquad
R_1R_2\subseteq R_2,\\
R_2R_2\subseteq R_3,\qquad
BR_3\subseteq R_3.
\end{gathered}
\]
Consequently, the equivalence relation $\rho$
with classes $R_1,R_2,R_3$ is a semiring congruence.
Its quotient is isomorphic to $S_7$, with
$R_1,R_2,R_3$ corresponding to $1,a,\infty$,
respectively.
Since $S_c(ab)\in\mathsf{V}(S_7)$,
we again obtain $S_c(ab)\in\mathcal V$.
\end{proof}

\begin{pro}\label{proabc}
Let $\mathcal V\leq\mathsf{V}(S_7^0)$.
Then $S_c(abc)\notin\mathcal V$ if and only if
$\mathcal V$ satisfies
\begin{equation}\label{eq:27}
(xyz)^2\approx xyz.
\end{equation}
\end{pro}

\begin{proof}
Suppose that $\mathcal V$ satisfies \eqref{eq:27}.
Since $S_c(abc)$ fails this identity,
$S_c(abc)\notin\mathcal V$.
Conversely, suppose that $\mathcal V$ fails
\eqref{eq:27}.
By \eqref{eq:8}, there exist $S\in\mathcal V$
and $a,b,c\in S$ such that
\[
abc<a^2b^2c^2.
\]
Let $B=\langle a,b,c\rangle$.
By \eqref{eq:12}, \eqref{eq:2} and distributivity,
every element of $B$ is a nonempty sum of expressions from
\[
\begin{gathered}
a,\ b,\ c,\ a^2,\ b^2,\ c^2,\ ab,\ ac,\ bc,\\
a^2b,\ ab^2,\ a^2c,\ ac^2,\ b^2c,\ bc^2,\\
a^2b^2,\ a^2c^2,\ b^2c^2,\\
abc,\ a^2bc,\ ab^2c,\ abc^2,\
a^2b^2c,\ a^2bc^2,\ ab^2c^2,\ a^2b^2c^2.
\end{gathered}
\]

\textbf{Case 1.}
$abc\neq a^2bc$, $abc\neq ab^2c$ and $abc\neq abc^2$.
Using \eqref{eq:8} and \eqref{eq:22}, one checks that
none of the listed expressions other than $abc$
is below $abc$.
Consequently, a sum of these expressions can equal
$abc$ only when all its summands are $abc$.
Multiplying an equality with value $a$ by $bc$,
or one with value $ab$ by $c$, and using symmetry,
gives the analogous statement for
$a,b,c,ab,ac,bc$.

Consider the subsets
\[
\begin{gathered}
R_1=\{a\},\qquad R_2=\{b\},\qquad R_3=\{c\},\\
R_4=\{ab\},\qquad R_5=\{ac\},\qquad R_6=\{bc\},\\
R_7=\{abc\},\qquad
R_8=B\setminus\{a,b,c,ab,ac,bc,abc\}.
\end{gathered}
\]
The preceding observation shows that these sets are
pairwise disjoint and nonempty;
in particular, $a^2b^2c^2\in R_8$.
It also gives
\[
R_8+B\subseteq R_8,\qquad R_8B\subseteq R_8.
\]
Sums of distinct singleton classes lie in $R_8$.
Apart from those obtained by commutativity,
the products outside $R_8$ are precisely
\[
\begin{aligned}
R_1R_2&=R_4,&
R_1R_3&=R_5,&
R_2R_3&=R_6,\\
R_1R_6&=R_7,&
R_2R_5&=R_7,&
R_3R_4&=R_7.
\end{aligned}
\]
Thus the equivalence relation $\rho$ with classes
$R_1,\ldots,R_8$ is compatible with both operations.
Hence $\rho$ is a semiring congruence and
\[
B/\rho\cong S_c(abc).
\]
Consequently, $S_c(abc)\in\mathcal V$.

\textbf{Case 2.}
At least one of
\[
abc=a^2bc,\qquad abc=ab^2c,\qquad abc=abc^2
\]
holds.
By symmetry, assume that $abc=a^2bc$.
Then
\[
a\cdot bc=a^2\cdot bc,
\qquad
a\cdot bc\neq(a\cdot bc)^2.
\]
Applying Case~2 in the proof of Proposition~\ref{proab}
to the generators $a$ and $bc$, we obtain a semiring
congruence $\rho$ on $\langle a,bc\rangle$ such that
\[
\langle a,bc\rangle/\rho\cong S_7.
\]
Thus $S_7\in\mathcal V$.
Since $S_c(abc)\in\mathsf{V}(S_7)$,
we again obtain $S_c(abc)\in\mathcal V$.
\end{proof}

\subsection{The six fibres and the main theorem}

From \cite[Corollary 2.5]{gmrz} we know that every subvariety of $\mathsf{V}(S^0_7)$
that contains $S_c(abc)$ is nonfinitely based.
To answer the finite basis problem for the subvarieties of $\mathsf{V}(S^0_7)$,
it is enough to
describe the subvarieties of $\mathsf{V}(S^0_7)$ that do not contain $S_c(abc)$.
Let $\mathcal{L}(\mathsf{V}(S_7^0))$ and $\mathcal{L}(\mathsf{V}(S_7))$
denote the subvariety lattices of $\mathsf{V}(S_7^0)$ and $\mathsf{V}(S_7)$, respectively.
It is natural to consider the following mapping
\[
\varphi: \mathcal{L}(\mathsf{V}(S_7^0)) \to \mathcal{L}(\mathsf{V}(S_7)), ~\mathcal{V} \mapsto \mathcal{V} \cap \mathsf{V}(S_7).
\]
Then $\varphi$ is surjective and so
\[
\mathcal{L}(S_7^0)=\bigcup_{\mathcal{W}\in \mathcal{L}(\mathsf{V}(S_7))}\varphi^{-1}(\mathcal{W}).
\]
Suppose that $\mathcal{V}$ is a subvariety of $\mathsf{V}(S_7^0)$. Then
\[
S_c(abc) \in \mathcal{V} \Leftrightarrow S_c(abc) \in \mathcal{V}\wedge\mathsf{V}(S_7)
\Leftrightarrow S_c(abc) \in \varphi(\mathcal{V}).
\]
Consequently, $\mathcal{V}$ contains $S_c(abc)$ if and only if $\varphi(\mathcal{V})$ contains $S_c(abc)$.
Therefore, $\mathcal{V}$ does not contain $S_c(abc)$ if and only if $\varphi(\mathcal{V})$ does not contain $S_c(abc)$.
From \cite{gmrz} we know that there are $6$ subvarieties of $\mathsf{V}(S_7)$
that do not contain $S_c(abc)$: the trivial variety $\mathbf{T}$, $\mathsf{V}(M_2)$, $\mathsf{V}(T_2)$,
$\mathsf{V}(S_c(ab))$, $\mathsf{V}(M_2, S_c(a))$ and $\mathsf{V}(M_2, S_c(ab))$.
This requires us to characterize $\varphi^{-1}(\mathcal{W})$ for
\[
\mathcal{W}\in
\begin{gathered}
\{\mathbf{T},\mathsf{V}(M_2),\mathsf{V}(T_2),
\mathsf{V}(S_c(ab)),\mathsf{V}(M_2,S_c(a)),
\mathsf{V}(M_2,S_c(ab))\}.
\end{gathered}
\]

\begin{pro}\label{profibres}
\hspace*{\fill}
\begin{itemize}
\item[$(1)$]
$\varphi^{-1}(\mathbf{T})
=\{\mathbf{T},\mathsf{V}(D_2)\}$;

\item[$(2)$]
$\varphi^{-1}(\mathsf{V}(M_2))
=\{\mathsf{V}(M_2),\mathsf{V}(M_2,D_2),
\mathsf{V}(M_2^0)\}$;

\item[$(3)$]
$\varphi^{-1}(\mathsf{V}(T_2))
=[\mathsf{V}(T_2),\mathsf{V}(T_2^0)]$;

\item[$(4)$]
$\varphi^{-1}(\mathsf{V}(S_c(ab)))
=[\mathsf{V}(S_c(ab)),\mathsf{V}(S_c(ab)^0)]$;

\item[$(5)$]
$\varphi^{-1}(\mathsf{V}(M_2,T_2))
=[\mathsf{V}(M_2,T_2),\mathsf{V}(M_2^0,T_2^0)]$;

\item[$(6)$]
$\varphi^{-1}(\mathsf{V}(M_2,S_c(ab)))
=[\mathsf{V}(M_2,S_c(ab)),
\mathsf{V}(M_2^0,S_c(ab)^0)]$.
\end{itemize}
\end{pro}

\begin{proof}
Throughout the proof, let
$\mathcal V\leq\mathsf{V}(S_7^0)$.

$(1)$ By the known subvariety lattice of
$\mathsf{V}(S_7)$,
\[
\varphi(\mathcal V)=\mathbf T
\quad\Longleftrightarrow\quad
M_2,T_2\notin\mathcal V.
\]
By Propositions~\ref{pro25032101}
and~\ref{pro25032102}, this is equivalent to
$\mathcal V$ satisfying \eqref{eq:1} and \eqref{eq:24}.
Together with commutativity, these identities define
$\mathsf{V}(D_2)$.
Thus $\mathcal V$ is either $\mathbf T$ or
$\mathsf{V}(D_2)$.

$(2)$ Similarly,
\[
\varphi(\mathcal V)=\mathsf{V}(M_2)
\quad\Longleftrightarrow\quad
M_2\in\mathcal V,\quad T_2\notin\mathcal V.
\]
By Proposition~\ref{pro25032101},
this is equivalent to
\[
\mathsf{V}(M_2)\leq\mathcal V
\leq\mathsf{V}(M_2^0).
\]
The assertion now follows from
Proposition~\ref{pro25032110}.

For the remaining cases, the preceding identities
and equational bases give
\[
\begin{aligned}
\mathcal V\models\eqref{eq:26}
&\quad\Longleftrightarrow\quad
\mathcal V\leq\mathsf{V}(M_2^0,T_2^0),\\
\mathcal V\models\eqref{eq:27}
&\quad\Longleftrightarrow\quad
\mathcal V\leq\mathsf{V}(S_c(ab)^0,M_2^0).
\end{aligned}
\]
Indeed, \eqref{eq:26} implies \eqref{eq:7}, and,
using \eqref{eq:23},
\[
xy+z
\approx (xy)^2+z
\succeq (xy)^2z^2
\succeq xyz.
\]
This gives the first equivalence by Lemma~\ref{m0t0}.
Likewise, under \eqref{eq:27},
\[
xyz+t
\approx (xyz)^2+t
\succeq (xyz)^2t^2
\succeq xyzt,
\]
which gives the second equivalence by
Proposition~\ref{proab0m0}.
The reverse implications follow directly from
the corresponding bases.

Moreover, adjoining \eqref{eq:24} to these two
relative bases gives, respectively,
$\mathsf{V}(T_2^0)$ and $\mathsf{V}(S_c(ab)^0)$.
For the former, observe that
\[
xy
\approx (xy)^2
\succeq (xy)^2z^2
\succeq xyz;
\]
 the latter follows from Proposition~\ref{proab0basis}.

$(3)$ By the known subvariety lattice of
$\mathsf{V}(S_7)$ and the exclusion criteria,
\[
\begin{aligned}
\varphi(\mathcal V)=\mathsf{V}(T_2)
&\Longleftrightarrow
T_2\in\mathcal V,\quad
M_2,S_c(ab)\notin\mathcal V\\
&\Longleftrightarrow
\mathsf{V}(T_2)\leq\mathcal V
\leq\mathsf{V}(T_2^0).
\end{aligned}
\]

$(4)$ Similarly,
\[
\begin{aligned}
\varphi(\mathcal V)=\mathsf{V}(S_c(ab))
&\Longleftrightarrow
S_c(ab)\in\mathcal V,\quad
M_2,S_c(abc)\notin\mathcal V\\
&\Longleftrightarrow
\mathsf{V}(S_c(ab))\leq\mathcal V
\leq\mathsf{V}(S_c(ab)^0).
\end{aligned}
\]

$(5)$ We have
\[
\begin{aligned}
\varphi(\mathcal V)=\mathsf{V}(M_2,T_2)
&\Longleftrightarrow
M_2,T_2\in\mathcal V,\quad
S_c(ab)\notin\mathcal V\\
&\Longleftrightarrow
\mathsf{V}(M_2,T_2)\leq\mathcal V
\leq\mathsf{V}(M_2^0,T_2^0).
\end{aligned}
\]

$(6)$ Finally,
\[
\begin{aligned}
\varphi(\mathcal V)=\mathsf{V}(M_2,S_c(ab))
&\Longleftrightarrow
M_2,S_c(ab)\in\mathcal V,\quad
S_c(abc)\notin\mathcal V\\
&\Longleftrightarrow
\mathsf{V}(M_2,S_c(ab))\leq\mathcal V
\leq\mathsf{V}(M_2^0,S_c(ab)^0).
\end{aligned}
\]
This completes the proof.
\end{proof}

We are now ready to state the main result.

\begin{thm}\label{thm:main}
The variety $\mathsf{V}(S_7^0)$ has exactly fifteen finitely
based subvarieties, namely
\[
\begin{gathered}
\mathbf T,\qquad \mathsf{V}(D_2),\\
\mathsf{V}(M_2),\qquad
\mathsf{V}(M_2,D_2),\qquad
\mathsf{V}(M_2^0),\\
\mathsf{V}(T_2),\qquad
\mathsf{V}(T_2,D_2),\qquad
\mathsf{V}(T_2^0),\\
\mathsf{V}(S_c(ab)),\\
\mathsf{V}(M_2,T_2),\qquad
\mathsf{V}(M_2,T_2,D_2),\\
\mathsf{V}(M_2^0,T_2),\qquad
\mathsf{V}(M_2,T_2^0),\qquad
\mathsf{V}(M_2^0,T_2^0),\\
\mathsf{V}(M_2,S_c(ab)).
\end{gathered}
\]
Every other subvariety of $\mathsf{V}(S_7^0)$ is
nonfinitely based.
\end{thm}

\begin{proof}
By \cite[Corollary~2.5]{gmrz}, every subvariety of
$\mathsf{V}(S_7^0)$ containing $S_c(abc)$ is nonfinitely
based. Hence it remains to consider the six fibres in
Proposition~\ref{profibres}.

The first and second fibres contain two and three finitely based
varieties, respectively. By Proposition~\ref{pro:three-varieties},
the third fibre contains three. By
Corollary~\ref{cor:twointervals}, the fourth fibre contains only
one finitely based variety. Proposition~\ref{pro:five-varieties}
shows that all five members of the fifth fibre are finitely
based. Finally, Corollary~\ref{cor:twointervals} shows that the
sixth fibre also contains only one finitely based variety.
Therefore the total number is
\[
2+3+3+1+5+1=15.
\]
These are exactly the varieties displayed above, and their
finite equational bases have been given in the preceding
propositions. All remaining members of the fourth and sixth
fibres are nonfinitely based by
Corollary~\ref{cor:twointervals}. This proves both assertions.
\end{proof}

\begin{cor}\label{cor:limit-subvarieties}
The varieties $\mathsf{V}(S_c(abc))$ and $\mathsf{V}(SR_6)$
are the only limit subvarieties of $\mathsf{V}(S_7^0)$.
\end{cor}

\begin{proof}
By \cite[Theorem~3.6]{rjzl} and \cite[Theorem~5.4]{lry},
$\mathsf{V}(S_c(abc))$ and $\mathsf{V}(SR_6)$ are limit varieties.
Let $\mathcal V$ be a limit subvariety of
$\mathsf{V}(S_7^0)$. If $S_c(abc)\in\mathcal V$, then
$\mathsf{V}(S_c(abc))\leq\mathcal V$. Since the former is
nonfinitely based, the minimality of $\mathcal V$ gives
$\mathcal V=\mathsf{V}(S_c(abc))$.

Suppose that $S_c(abc)\notin\mathcal V$. By
Proposition~\ref{profibres}, $\mathcal V$ belongs to one of the six
fibres. The first, second, third and fifth fibres contain only
finitely based varieties. Hence $\mathcal V$ belongs to the fourth
or the sixth fibre and is not its least member. Proposition
~\ref{prop:strictcontains} yields
$\mathsf{V}(SR_6)\leq\mathcal V$. The minimality of $\mathcal V$
now gives $\mathcal V=\mathsf{V}(SR_6)$.
\end{proof}

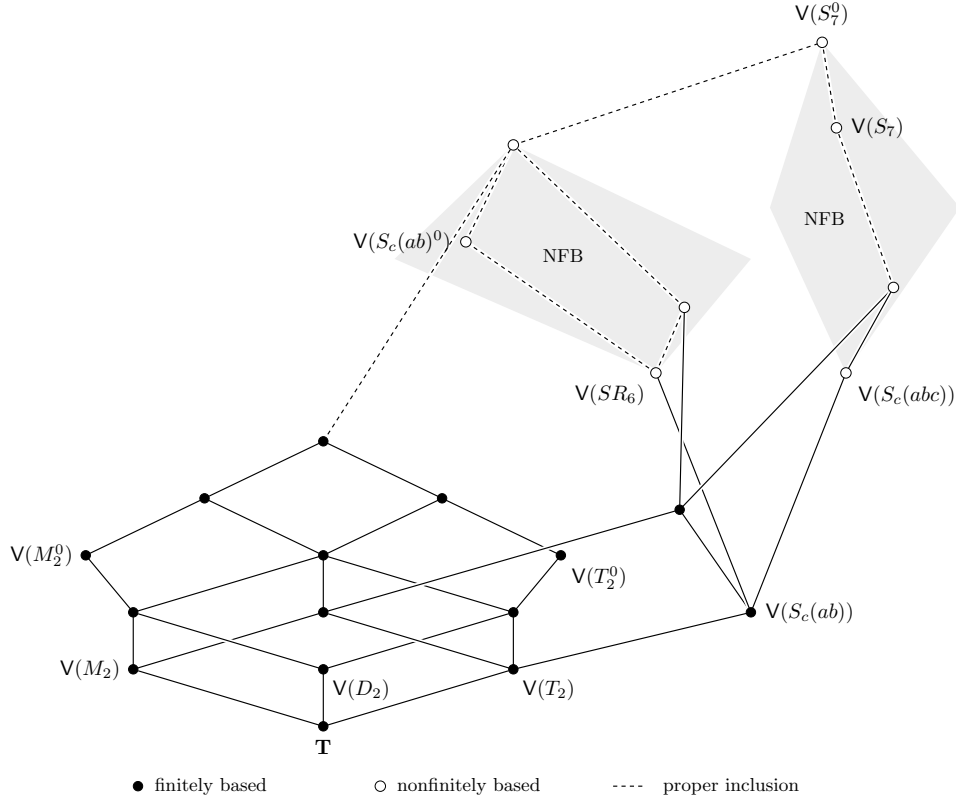
\begin{figure}[p]
\centering
\resizebox{\textwidth}{!}{%
\begin{tikzpicture}[x=0.8cm,y=0.96cm,
  every node/.style={font=\normalsize},
  lab/.style={inner sep=3pt},
  fb/.style={circle,draw,fill=black,inner sep=0pt,minimum size=4.5pt},
  nfb/.style={circle,draw,fill=white,inner sep=0pt,minimum size=5pt},
  edge/.style={line width=0.55pt,preaction={draw=white,line width=2.3pt}},
  inc/.style={edge,dash pattern=on 2pt off 2pt}]

\coordinate (triv) at (0,0);
\coordinate (M) at (-4,1); \coordinate (D) at (0,1);
\coordinate (T) at (4,1); \coordinate (MD) at (-4,2);
\coordinate (MT) at (0,2); \coordinate (TD) at (4,2);
\coordinate (M0) at (-5,3); \coordinate (B) at (0,3);
\coordinate (T0) at (5,3); \coordinate (C) at (-2.5,4);
\coordinate (E) at (2.5,4); \coordinate (I) at (0,5);
\coordinate (A) at (9,2); \coordinate (MA) at (7.5,3.8);
\coordinate (R) at (7,6.2); \coordinate (RM) at (7.6,7.35);
\coordinate (A0) at (3,8.5); \coordinate (W) at (4,10.2);
\coordinate (C3) at (11,6.2); \coordinate (MC3) at (12,7.7);
\coordinate (S7) at (10.8,10.5); \coordinate (S70) at (10.5,12);

\fill[black!7] (R)--(1.5,8.2)--(W)--(9,8.2)--cycle;
\fill[black!7] (C3)--(9.4,9.1)--(S70)--(13.4,9.1)--cycle;
\foreach \a/\b in {R/RM,R/A0,RM/W,A0/W,I/W,W/S70,MC3/S7,S7/S70}
  \draw[inc] (\a)--(\b);
\foreach \a/\b in {triv/M,triv/D,triv/T,M/MD,M/MT,D/MD,D/TD,
  T/MT,T/TD,MD/M0,MD/B,TD/T0,TD/B,MT/B,M0/C,T0/E,B/C,B/E,
  C/I,E/I,T/A,A/MA,MT/MA,A/R,MA/RM,A/C3,MA/MC3,C3/MC3}
  \draw[edge] (\a)--(\b);

\foreach \a in {triv,M,D,T,MD,MT,TD,M0,B,T0,C,E,I,A,MA}
  \node[fb] at (\a) {};
\foreach \a in {R,RM,A0,W,C3,MC3,S7,S70}
  \node[nfb] at (\a) {};

\node[lab,below=3pt] at (triv) {$\mathbf T$};
\node[lab,left=3pt] at (M) {$\mathsf{V}(M_2)$};
\node[lab,below right=2pt] at (D) {$\mathsf{V}(D_2)$};
\node[lab,below right=2pt] at (T) {$\mathsf{V}(T_2)$};
\node[lab,left=3pt] at (M0) {$\mathsf{V}(M_2^0)$};
\node[lab,below right=3pt] at (T0) {$\mathsf{V}(T_2^0)$};
\node[lab,right=4pt] at (A) {$\mathsf{V}(S_c(ab))$};
\node[lab,below left=3pt] at (R) {$\mathsf{V}(SR_6)$};
\node[lab,left=4pt] at (A0) {$\mathsf{V}(S_c(ab)^0)$};
\node[lab,below right=4pt] at (C3) {$\mathsf{V}(S_c(abc))$};
\node[lab,right=4pt] at (S7) {$\mathsf{V}(S_7)$};
\node[lab,above=5pt] at (S70) {$\mathsf{V}(S_7^0)$};

\node[align=center,font=\small] at (5.05,8.25) {NFB};
\node[align=center,font=\small] at (10.55,8.9) {NFB};
\node[fb] at (-3.9,-1.05) {};
\node[anchor=west,font=\small] at (-3.7,-1.05) {finitely based};
\node[nfb] at (1.2,-1.05) {};
\node[anchor=west,font=\small] at (1.4,-1.05) {nonfinitely based};
\draw[inc] (6.1,-1.05)--(6.8,-1.05);
\node[anchor=west,font=\small] at (7,-1.05) {proper inclusion};
\end{tikzpicture}%
}
\caption{The finite basis classification of the subvarieties of
$\mathsf{V}(S_7^0)$. The fifteen filled points are precisely the
finitely based subvarieties. The shaded nonfinitely based intervals
are not expanded, and unlabelled points represent joins. Crossings
without a point are not vertices.}
\label{fig:subvariety-lattice}
\end{figure}
\clearpage


\bibliographystyle{amsplain}

\end{document}